\documentclass[a4paper, 12pt]{amsart}
\usepackage{amsmath,amssymb,amsthm}
\usepackage{amscd}
\usepackage{array,enumerate}
\newtheorem{Thm}{Theorem}[section]
\newtheorem{Prop}[Thm]{Proposition}
\newtheorem{Lem}[Thm]{Lemma}
\newtheorem{Cor}[Thm]{Corollary}
\newtheorem{Thmint}{Theorem}[section]

\newtheorem{MThm}{}

\theoremstyle{definition}
\newtheorem{Rem}[Thm]{Remark}

\newtheorem{Def}[Thm]{Definition}

\newcommand{\Cs}{C$^\ast$}

\newcommand{\sd}{^{\ast\ast}}
\newcommand{\id}{\mbox{\rm id}}

\newcommand{\rg}{\mathop{{\mathrm C}_{\mathrm r}^\ast}}

\newcommand{\rc}{\mathop{\rtimes _{\mathrm r}}}

\newcommand{\ac}{\mathop{\rtimes _{\mathrm{alg}}}}

\newcommand{\rca}[1]{\mathop{\rtimes _{{\mathrm r}, #1}}}

\DeclareMathOperator{\supp}{supp}

\DeclareMathOperator{\bigfp}{\lower0.25ex\hbox{\LARGE $\ast$}}
\newcommand{\ad}{\mathop{\rm ad}}
\title[Extremely tight \Cs-inclusions]
{Non-amenable tight squeezes by Kirchberg algebras}
\author{Yuhei Suzuki}
\subjclass[2010]{Primary~
46L55, Secondary~46L05, 
46L07}
\keywords{Kirchberg algebras, \Cs-dynamical systems, rigid inclusions}
\address{Department of Mathematics, Faculty of Science, Hokkaido University,
Kita 10, Nishi 8, Kita-Ku, Sapporo, Hokkaido, 060-0810, Japan}
\email{yuhei@math.sci.hokudai.ac.jp}
\begin{document}

\begin{abstract}
We give a framework to produce \Cs-algebra inclusions with extreme properties.
This gives the first constructive nuclear minimal ambient \Cs-algebras.
We further obtain a purely infinite analogue of Dadarlat's modeling theorem on AF-algebras:
Every Kirchberg algebra is rigidly and KK-equivalently sandwiched by non-nuclear \Cs-algebras without intermediate \Cs-algebras.
Finally we reveal a novel property of Kirchberg algebras:
They embed into arbitrarily wild \Cs-algebras
as rigid maximal \Cs-subalgebras.
\end{abstract}
\maketitle
\section{Introduction}
Thanks to the recent progress in the classification theory of amenable \Cs-algebras,
Elliott's classification program has been almost completed; see \cite{Win} for a recent survey.
As crucial ideas and techniques have been developed in this theory
(see e.g., \cite{EGLN}, \cite{Kir}, \cite{MS2}, \cite{Phi}, \cite{TWW}, and references in \cite{Win}),
a next natural attempt is applying the theory and its byproducts to understand the structure of simple \Cs-algebras beyond the classifiable class (e.g., the reduced non-amenable group \Cs-algebras).

To take the advantage of rich structures of classifiable \Cs-algebras to understand non-amenable \Cs-algebras, one possible natural strategy
is to bridge two \Cs-algebras from each class via a tight inclusion.
Indeed tight inclusions of operator algebras receive much attentions
and are deeply studied by many hands because of their importance in the structure theory of operator algebras: see e.g., \cite{Ham79}, \cite{Ham85}, \cite{KK}, \cite{Lon}, \cite{Oza07}, \cite{Pop81}, \cite{Pop}, \cite{PopICM}.
Recent highlights are the breakthrough results on \Cs-simplicity
\cite{KK}, \cite{BKKO} (cf.~ \cite{Ham85}), in which
tight inclusions are used to reduce problems on the reduced group \Cs-algebras
to those of less complicated \Cs-algebras.
This suggests the existence of \emph{boundary theory} for general \Cs-algebras (cf.~\cite{Oza07}).
Other strategies based on expansions of \Cs-algebras
also work successfully in the Baum--Connes conjecture \cite{BCH}, see e.g., \cite{Hig}, \cite{HK}.
These significant results are motivations behind the present work.

The purpose of the present paper is to establish
a new powerful framework to produce extreme examples of tight \Cs-algebra inclusions.
We particularly consider the following three conditions coming from the three different viewpoints:
\begin{description}
\item[Algebraic side]
absence of intermediate \Cs-algebras (maximality/minimality),
\item[Topological side] KK-equivalence,
\item[Order structural side] Hamana's operator system rigidity \cite{Ham79b}. 
\end{description}
We note that the third condition is a crucial ingredient of \cite{KK}, \cite{BKKO}.
Several questions on the first condition were posed by Ge \cite{Ge} for instance.
The second condition is partly motivated by the Baum--Connes conjecture (cf.~ \cite{Hig}, \cite{HK}).

We now present the Main Theorem of this paper.
Throughout the paper, denote by $\mathbb{F}_\infty$
a countable free group of infinite rank.
\begin{MThm}[see Proposition \ref{Prop:upert2}, Theorems \ref{Thm:inter}, \ref{Thm:rigid}]
Let $A$ be a simple unital separable purely infinite \Cs-algebra.
Let $\alpha \colon \mathbb{F}_\infty \curvearrowright A$ be an approximately inner \Cs-dynamical system.
Then there is an inner perturbation $\gamma$ of $\alpha$
with the following property:
Any \Cs-dynamical system $\beta \colon \mathbb{F}_\infty \curvearrowright B$
on a simple \Cs-algebra $B$ with $B^\beta \neq 0$ gives a rigid inclusion
$B \rca{\beta}\mathbb{F}_\infty \subset (A\otimes B) \rca{\gamma\otimes \beta} \mathbb{F}_\infty$ without intermediate \Cs-algebras.
\end{MThm}
Note that the statement does not exclude the case that $B= \mathbb{C}$
and the case that $\alpha$ is the trivial action.
Even in these specific cases, the theorem provides many \Cs-algebra inclusions
of remarkable new features.
The Main Theorem also sheds some light on \Cs-dynamical systems.
Inner automorphisms are usually regarded as trivial objects
in the study of automorphisms of (single) \Cs-algebras.
In fact when two \Cs-dynamical systems are \emph{cocycle conjugate},
their crossed product \Cs-algebras are isomorphic.
However, the Main Theorem reveals that the associated
\Cs-algebra inclusions can be changed drastically
by inner automorphisms.
It is also interesting to compare these phenomena with the remarkable rigidity phenomena on the crossed product algebra inclusions studied and conjectured by Neshveyev--St{\o}rmer \cite{NS} (see also the recent works \cite{CD}, \cite{Suz19}).

\subsection*{Applications to Kirchberg algebras}
Since it is fairly easy to construct free group \Cs-dynamical systems (because of the freeness),
the Main Theorem has a wide range of applications.
Furthermore, although it is not immediately apparent from the statement, the Main Theorem is
successfully applied to \emph{arbitrary} Kirchberg algebras.
As a consequence, we obtain novel properties of Kirchberg algebras.

Recall that a \Cs-algebra is said to be a \emph{Kirchberg algebra} if it is simple, separable, nuclear, and purely infinite. We refer the reader to the book \cite{Ror} for basic facts and backgrounds on Kirchberg algebras.
Some beautiful and rich features of Kirchberg algebras
can be seen from the complete classification theorem of Kirchberg \cite{Kir} and Phillips \cite{Phi}.
We also refer the reader to \cite{IM}, \cite{IM2} (cf.~\cite{DP}) for a new interaction between algebraic topology
and the symmetry structure of Kirchberg algebras.

As the first main consequence, we obtain the first constructive examples of nuclear minimal ambient \Cs-algebras.
(Here \emph{constructive} at least means that all constructions are
elementary and concretely understandable and avoid the Baire category theorem.)
\begin{Thmint}\label{Thmint:Main}
Let $\alpha \colon \mathbb{F}_\infty \curvearrowright A$ be a \Cs-dynamical system on a simple separable nuclear \Cs-algebra with $A^\alpha \neq 0$.
Then the reduced crossed product $A \rca{\alpha} \mathbb{F}_\infty$
admits a KK-equivalent rigid embedding into a Kirchberg algebra without intermediate \Cs-algebras.
\end{Thmint}
Note that, in our previous work \cite{Suzmin}, we obtained partial results in specific cases (without KK-condition), which in particular gave
the first examples of nuclear minimal ambient \Cs-algebras.
However all constructions in \cite{Suzmin} depend on the Baire category theorem (applied to the space of Cantor systems of $\mathbb{F}_\infty$).
It is a novelty of the present paper that our new constructions
are quite elementary and avoid the Baire category theorem.

The second main consequence is the following structure result on Kirchberg algebras.
This theorem can be considered as a Kirchberg algebra analogue of Dadarlat's modeling theorem for AF-algebras \cite{Da}.
Since Kirchberg algebras have no elementary inductive limit structure,
our approach naturally differs from \cite{Da}.
\begin{Thmint}\label{Thmint:Kir}
For every Kirchberg algebra $A$, there are \Cs-algebra inclusions
$B \subset A \subset C$ satisfying the following conditions.
\begin{itemize}
\item $B$ is non-nuclear, $C$ is non-exact, and both algebras are simple and purely infinite.
\item The \Cs-subalgebras $B\subset A$ and $A \subset C$ are rigid and maximal. 
\item These inclusions give KK-equivalences.
\end{itemize}
\end{Thmint}
As the third application, we obtain the following ubiquitous property for some \Cs-algebras.
Although the statement holds true for more general \Cs-algebras,
we concentrate on the particularly interesting case.
\begin{Thmint}\label{Thmint:3}
Let $A$ be a unital Kirchberg algebra.
Then for any unital separable \Cs-algebra $B$, there exists an ambient \Cs-algebra $C$ of $B$
with a faithful conditional expectation which also contains $A$ as a rigid maximal \Cs-subalgebra.
\end{Thmint}
 It is easy to see from the proof that $\rg(\mathbb{F}_\infty)$ also has the same property.
Moreover one can show that the \emph{free group factor} $L(\mathbb{F}_\infty)$
has the analogous property by a similar method (Remark \ref{Rem:LF}).
It would be interesting to ask if the other free group factors $L(\mathbb{F}_n)$; $n=2, 3, \ldots$ have the same property.

A key ingredient of the proofs is amenable actions of $\mathbb{F}_\infty$ on Kirchberg algebras \cite{Suzeq}, \cite{Suz19}.
The existence of amenable actions (of non-amenable groups) on \emph{simple} \Cs-algebras does not seem to have been believed for a long time (cf.~\cite{AD02}, \cite{BO}) until \cite{Suzeq}.
(It is notable that such actions do \emph{not} exist in the von Neumann algebra context; see \cite{AD79}, Corollary 4.3.)
To exclude intermediate \Cs-algebras of the reduced crossed product inclusions,
as stated in the Main Theorem, we perturb actions by inner automorphisms.
In contrast to commutative \Cs-algebras, purely infinite simple \Cs-algebras have sufficiently many
inner automorphisms (by Cuntz's result \cite{Cun}; see Lemma 2.4 below).
This provides
amenable actions on Kirchberg algebras which sufficiently mix projections.
(Note that by Zhang's theorem \cite{Zha}, purely infinite simple \Cs-algebras are of real rank zero. In particular they have plenty of nontrivial projections. This theorem plays a prominent role throughout the article.)
Another important advantage of using Kirchberg algebras
is their freedom of K-theory:
in contrast to the fact that the K-groups of compact spaces
are restricted (for instance their K$^0$-groups must have a non-trivial order structure), there is no structural restriction on the K-groups of Kirchberg algebras.
This leads to useful reduced crossed product decompositions of
Kirchberg algebras by $\mathbb{F}_\infty$ (up to stable isomorphism) \cite{Suz19}.

\subsection*{Organization of the paper}
In Section \ref{Sec:inn}, we develop techniques on inner perturbations of \Cs-dynamical systems.
This provides \Cs-dynamical systems with extremely transitive properties.
In Section \ref{Sec:proof}, we study restrictions of intermediate objects of certain structures
associated with \Cs-dynamical systems obtained in Section \ref{Sec:inn}.
In Section \ref{Sec:rigidity},
we discuss rigidity properties of inclusions.
In particular, after improving the constructions of inner perturbations in Section \ref{Sec:inn},
we complete the proof of the Main Theorem.
Finally, in Section \ref{Sec:Kir}, we prove the consequences of the Main Theorem (Theorems \ref{Thmint:Main} to \ref{Thmint:3}).

To handle the non-unital case, we need technical results,
which we discuss in the Appendix.
As a byproduct, we extend
the tensor splitting theorem \cite{Zac}, \cite{Zsi} (cf.~\cite{GK}) to the non-unital case.

Finally, we remark that, except for some cases, the crossed product splitting theorem obtained in \cite{Suz19} is \emph{not available} because of the failure of central freeness.
Our method to exclude intermediate \Cs-algebras
is a sophisticated version of the argument developed in our previous work \cite{Suzmin}. Because of non-commutativity and K-theoretic obstructions,
we need technical improvements.

For basic facts on \Cs-algebras and discrete groups, we refer the reader to the book \cite{BO}.
For basic facts on K-theory and KK-theory, see the book \cite{Bla}.
\subsection*{Notations}Here we fix some notations.
Notations not explained in the article should be very common in operator algebra theory.
\begin{itemize}
\item For $\epsilon>0$ and for two elements $x$, $y$ of a \Cs-algebra,
denote by $x\approx_{\epsilon} y$ if $\|x -y\| <\epsilon$.
\item The symbols `$\otimes$', `$\rc$', `$\rtimes_{\rm alg}$' stand
for the minimal tensor products (of \Cs-algebras and completely bounded maps) and the reduced \Cs- and algebraic crossed products respectively.

\item For a \Cs-algebra $A$, denote by $A^{\rm p}$, $A_+$
the set of projections and the cone of positive elements in $A$ respectively.
\item For a unital \Cs-algebra $A$, denote by $A^{\rm u}$
the group of unitary elements in $A$.
\item For a \Cs-algebra $A$, denote by $\mathcal{M}(A)$, $Z(A)$, $A\sd$
the multiplier algebra of $A$, the center of $A$, and the second dual of $A$ respectively.
\item When there is an obvious \Cs-algebra embedding $A \rightarrow \mathcal{M}(B)$,
we regard $A$ as a \Cs-subalgebra of $\mathcal{M}(B)$ via the obvious embedding.
Such a situation often occurs in the tensor product, the free product, and the crossed product constructions.
\item For the reduced crossed product $A \rc \Gamma$ and $s\in \Gamma$, denote by $u_s$ its canonical implementing unitary element in $\mathcal{M}(A \rc \Gamma)$.
\item For the reduced crossed product $A \rc \Gamma$,
denote by $E\colon A\rc \Gamma \rightarrow A$ the conditional expectation
satisfying $E(a u_s)=0$ for all $a\in A$ and $s\in \Gamma \setminus \{e\}$ (called the canonical conditional expectation).
\item For a \Cs-dynamical system $\alpha \colon \Gamma \curvearrowright A$,
denote by $A^\alpha$ the fixed point algebra of $\alpha$:
\[A^\alpha:=\{ a\in A: \alpha_s(a)=a {\rm~for~all~}s\in \Gamma\}.\]
\item For two \Cs-dynamical systems $\alpha \colon \Gamma \curvearrowright A$ and $\beta \colon \Gamma \curvearrowright B$,
denote by $\alpha \otimes \beta$ the diagonal action of $\alpha$ and $\beta$, 
that is, the action $\Gamma \curvearrowright A \otimes B$ defined to be $(\alpha \otimes \beta)_s := \alpha_s \otimes \beta_s$ for $s\in \Gamma$.
\item For a \Cs-algebra $A$, 
denote by $1$ the unit of $\mathcal{M}(A)$.
\item For a unital \Cs-algebra $A$, denote by $\mathbb{C}$ the subspace of $A$ spanned by $1$.
\end{itemize}

\section{Inner perturbations of \Cs-dynamical systems}\label{Sec:inn}
In this section, we develop techniques on \emph{inner perturbations} of
free group \Cs-dynamical systems.
This provides \Cs-dynamical systems with an extreme transitivity;
see Proposition \ref{Prop:upert}.
The results in this section play crucial roles in the proof of the Main Theorem.

We first introduce the following metric spaces of projections in a \Cs-algebra.
\begin{Def}
Let $A$ be a \Cs-algebra.
For $x_1, x_2 \in {\mathrm K}_0(A)$,
we define
\[{\rm P}(A; x_1, x_2):= \left\{(p_1, p_2) \in (A^{\rm p} \setminus \{0 \})^2:
[p_1]_0=x_1,\ [p_2]_0=x_2,\ p_1 \perp p_2,\ p_1 + p_2 \neq 1\right\}\]
(possibly empty).
We equip ${\rm P}(A; x_1, x_2)$ with the metric given by the \Cs-norm on $A\oplus A$.
\end{Def}
Note that each ${\rm P}(A; x_1, x_2)$ is closed in $A\oplus A$.
Observe that every automorphism $\alpha$ on $A$ which acts trivially on the K$_0$-group induces an isometric homeomorphism
\[(p_1, p_2) \mapsto (\alpha(p_1), \alpha(p_2))\]
 on each ${\rm P}(A; x_1, x_2)$.

In this paper we employ the following definition of amenability for \Cs-dynamical systems which is introduced by Anantharaman-Delaroche \cite{AD}.
\begin{Def}[see \cite{AD}, Definition 4.1]\label{Def:ame}
A \Cs-dynamical system $\alpha \colon \Gamma \curvearrowright A$ is said to be \emph{amenable}
if the induced action $\Gamma \curvearrowright Z(A\sd)$ is amenable in the von Neumann algebra sense.
\end{Def}
Although there is another property of \Cs-dynamical systems called amenable (see e.g., \cite{BO}),
in this paper amenability always means the property in Definition \ref{Def:ame} without specified.
In this paper, we do not use the definition directly, but use the following facts.
\begin{enumerate}
\item
It is clear from the definition that when one of \Cs-dynamical systems $\alpha$, $\beta$ of $\Gamma$ is amenable, so is $\alpha \otimes \beta$.
\item When the underlying \Cs-algebra is nuclear,
amenability of \Cs-dynamical systems is equivalent to the nuclearity of the reduced crossed product
(\cite{AD}, Theorem 4.5).
\end{enumerate}

We give one more basic property of amenability.
This immediately follows from the definition,
but plays an important role in this paper.
Before giving the statement, we recall and introduce a few definitions.

Recall that an automorphism $\alpha$ of a \Cs-algebra $A$
is said to be \emph{inner} if there exists $u\in \mathcal{M}(A)^{\rm u}$
satisfying $\alpha(x)=\ad(u)(x):=uxu^\ast$ for all $x\in A$.
Denote by ${\rm Inn}(A)$ the group of inner automorphisms of $A$.
For two \Cs-dynamical systems $\alpha, \beta \colon \Gamma \curvearrowright A$, we say
that \emph{$\beta$ is an inner perturbation of $\alpha$} if
$\beta_s \circ \alpha_{s}^{-1} \in {\rm Inn}(A)$ for all $s\in \Gamma$.
Note that, as ${\rm Inn}(A)$ forms a normal subgroup in the automorphism group of $A$, to check that $\beta$ is an inner perturbation of $\alpha$,
we only need to confirm the condition on a generating set of $\Gamma$.
\begin{Lem}\label{Lem:ame}
Amenability of \Cs-dynamical systems is stable under inner perturbations.
\end{Lem}
\begin{proof}
Inner automorphisms on a \Cs-algebra $A$ induce the identity map on $Z(A\sd)$.
\end{proof}
We remark that ${\rm Inn}(A)$ in Lemma \ref{Lem:ame} is not replaceable by
its pointwise norm closure (\emph{the group of approximately inner automorphisms}).
In fact, when the acting group is a non-commutative free group, by \cite{KOS},
any \Cs-dynamical system on
a simple separable \Cs-algebra admits a non-amenable approximately inner
 perturbation (see the Proposition in \cite{Suzfp} for details and an application).

We next recall the following basic observation.
The proof is essentially contained in \cite{Cun} and should be well-known but we include it for completeness.
\begin{Lem}\label{Lem:trans}
Let $A$ be a purely infinite simple \Cs-algebra.
Then for any $x_1, x_2 \in {\mathrm K}_0(A)$,
the induced action ${\rm Inn}(A) \curvearrowright {\rm P}(A; x_1, x_2)$
is transitive.
\end{Lem}
\begin{proof}
Note first that in the unital case, the statement follows from \cite{Cun}, Section 1.
To consider the non-unital case, we first show that each orbit of the action is open in ${\rm P}(A; x_1, x_2)$.
Let $(p_1, p_2), (q_1, q_2) \in {\rm P}(A; x_1, x_2)$ be given.
Assume that $\|p_1 - q_1\|, \|p_2 -q_2\|<1/12$.
By Lemma 7.2.2 in \cite{BO},
there is $u \in \mathcal{M}(A)^{\rm p}$
with $q_1=up_1u^\ast$, $\|1-u\|< 1/3$.
This implies $\|up_2u^\ast - q_2\|< 1$.
Applying Lemma 7.2.2 in \cite{BO} to 
the projections $up_2u^\ast$, $q_2$ in the \Cs-algebra $(1-q_1)\mathcal{M}(A)(1-q_1)$,
we obtain $v\in \mathcal{M}(A)^{\rm u}$
with $vup_2u^\ast v^\ast =q_2$, $vq_1=q_1$.
Now it is clear that $\ad(vu)(p_i) =q_i$ for $i=1, 2$.

By \cite{Zha}, $A$ admits a (not necessary increasing) approximate unit $(e_j)_{j \in J}$ consisting of projections.
Thus for any $(p_1, p_2)\in {\rm P}(A; x_1, x_2)$, by standard applications of functional calculus
and the observation in the previous paragraph,
for any sufficiently large $j \in J$, one can find $\alpha \in {\rm Inn}(A)$
with $\alpha(p_1+p_2) \lneq e_j$.
This reduces the proof to the unital case,
and thus completes the proof.
\end{proof}
Now we are able to show the following result.
We say that a group action $\Gamma \curvearrowright X$
on a topological space is \emph{minimal}
if all $\Gamma$-orbits are dense in $X$.
\begin{Prop}\label{Prop:upert}
Let $\alpha \colon \mathbb{F}_\infty \curvearrowright A$ be a \Cs-dynamical system on
a separable purely infinite simple \Cs-algebra $A$ whose induced action on ${\mathrm K}_0(A)$ is trivial.
Then there exists an inner perturbation
$\beta$ of $\alpha$ satisfying the following conditions.
\begin{enumerate}[\upshape(1)]
\item For any $x_1, x_2 \in {\mathrm K}_0(A)$, the induced action
$\mathbb{F}_\infty \curvearrowright {\rm P}(A; x_1, x_2)$ of $\beta$
is minimal.
\item Let $S_\beta$ denote the set of all $p \in A^{\rm p}$ whose stabilizer subgroup of $\beta$ contains at least two
canonical generating elements of $\mathbb{F}_\infty$.
Then $S_\beta$ is dense in $A^{\rm p}$ in norm.
\end{enumerate}
\end{Prop}
\begin{proof}
Since any automorphism on $A$ which acts trivially on K$_0(A)$ induces an isometric
homeomorphism on each ${\rm P}(A; x_1, x_2)$,
to check condition (1),
we only need to find a dense orbit
in each ${\rm P}(A; x_1, x_2)$.

For each $x_1, x_2 \in {\mathrm K}_0(A)$,
choose a dense sequence $(p[x_1, x_2, n, 1], p[x_1, x_2, n, 2])_{n =1}^\infty$ in $ {\rm P}(A; x_1, x_2)$ such that each term appears at least twice in the sequence.
Denote by $S$ the canonical generating set of $\mathbb{F}_\infty$.
We fix a bijective map
\[f \colon {\mathrm K}_0(A) \times {\mathrm K}_0(A) \times \mathbb{N}\times \{1, 2\} \rightarrow S.\]
For each $(x_1, x_2, n) \in {\mathrm K}_0(A) \times {\mathrm K}_0(A) \times \mathbb{N}$,
choose $u[x_1, x_2, n, i] \in \mathcal{M}(A)^{\rm u}$; $i=1, 2$
satisfying
\[(\ad(u[x_1, x_2, n, 1])\circ \alpha_{f(x_1, x_2, n, 1)})(p[x_1, x_2, 1, j])= p[x_1, x_2, n, j],\]
\[(\ad(u[x_1, x_2, n, 2]) \circ\alpha_{f(x_1, x_2, n, 2)})(p[x_1, x_2, n, j])= p[x_1, x_2, n, j],\]
for $j=1, 2$ (this is possible by Lemma \ref{Lem:trans}).
For $s\in S$,
define
\[\beta_{s}:=\ad(u[f^{-1}(s)]) \circ \alpha_{s}.\]
This formula defines an inner perturbation $\beta$ of $\alpha$.
It is clear from the choice of $u$'s that $\beta$ satisfies the required conditions.
\end{proof}
\section{Transitivity conditions and absence of intermediate objects}\label{Sec:proof}
In this section, we use conditions (1) and (2) in Proposition \ref{Prop:upert} to exclude certain intermediate objects.
These two conditions can be seen as a non-commutative variant
of the property $\mathcal{R}$ defined for Cantor systems in \cite{Suzmin}, Proposition 3.3.
On the one hand, because the property $\mathcal{R}$ requires an extreme transitivity (seemingly opposite to amenability of topological dynamical systems), it seems hopeless to obtain a constructive amenable example.
On the other hand, in contrast to this, we have already obtained
constructive amenable \Cs-dynamical systems satisfying these two conditions, thanks to high non-commutativity of the underlying algebras.

A ($\mathbb{C}$-linear) subspace $X$ of a \Cs-algebra is said to be \emph{self-adjoint}
if $x^\ast \in X$ for all $x\in X$.
For a \Cs-dynamical system $\alpha \colon \Gamma \curvearrowright A$,
a subspace $X \subset A$ is said to be \emph{$\alpha$-invariant}
if it satisfies $\alpha_s(X)=X$ for all $s\in \Gamma$.

To study intermediate \Cs-algebras, we first show that, when the underlying algebra is simple and purely infinite, from condition (1),
we obtain the best possible restriction on invariant closed subspaces.
The reason why we need to study subspaces rather than just subalgebras is as follows.
For a \Cs-subalgebra $C$ of the reduced crossed product $A \rca{\alpha} \Gamma$
satisfying $u_sCu_s^\ast =C$ for all $s\in \Gamma$,
the set $E(C)$ forms an $\alpha$-invariant self-adjoint subspace of $A$, but
it is \emph{not necessary a subalgebra}.

\begin{Prop}\label{Prop:invsp}
Let $A$ be a purely infinite simple \Cs-algebra.
Let $\alpha \colon \mathbb{F}_\infty \curvearrowright A$ be a \Cs-dynamical system
which acts trivially on ${\mathrm K}_0(A)$
and satisfies condition $(1)$ in Proposition \ref{Prop:upert}.
Then $0$, $\mathbb{C}$, $A$
are the only possible $\alpha$-invariant closed
self-adjoint subspaces of $A$.
\end{Prop}
\begin{proof}
Let $X$ be an $\alpha$-invariant closed self-adjoint subspace of $A$
different from $0$ and $\mathbb{C}$.
It suffices to show that $X=A$.
We first note that, in the unital case, the subspace $X+\mathbb{C}$ is also $\alpha$-invariant, closed, and self-adjoint.
We observe that the equality $X+\mathbb{C} =A$ implies $X=A$.
To see this, assume that $X+\mathbb{C}=A$, $X \neq A$, and denote by $\varphi$
the (nonzero bounded) linear functional on $A$ defined by
the (Banach space) quotient map $A \rightarrow A/X \cong \mathbb{C}$.
Then, since $X$ is $\alpha$-invariant, so is $\varphi$ (that is, $\varphi \circ \alpha_s=\varphi$ for all $s\in \mathbb{F}_\infty$).
It follows from condition (1) that for any two projections $p_1, p_2 \in A^{\rm p}\setminus \{0, 1\}$ with $[p_1]_0 =[p_2]_0$,
we have $\varphi(p_1)=\varphi(p_2)$.
Since 
$\varphi(1)\neq 0$, one can find
$p\in A^{\rm p}\setminus \{0, 1\}$ with
$\varphi(p)\neq 0$.
Choose pairwise orthogonal nonzero projections $(p_n)_{n=1}^\infty$ in $A$
satisfying $[p_n]_0 =[p]_0$ for all $n\in \mathbb{N}$.
Then for any $N\in \mathbb{N}$,
we have
$N|\varphi(p)|= |\varphi(\sum_{n=1}^N p_n)|\leq \|\varphi\|$.
This is a contradiction.
Thus, in the unital case, we only need to show the statement
under the additional assumption that $\mathbb{C} \subsetneq X$.

We will prove $ A^{\rm p} \subset X$ under this assumption, which
implies $X=A$ by \cite{Zha}.
Choose a self-adjoint contractive element $h$ in $X \setminus \mathbb{C}$
whose spectrum contains $0$ and $1$.
Let $\epsilon>0$ and $p \in A^{\rm p} \setminus \{0, 1\}$ be given.
Since $A$ is of real rank zero \cite{Zha} (see also \cite{BP}), there exist nonzero pairwise orthogonal projections
$p_1, \ldots, p_l$ in $A$ and a sequence $\lambda_1, \ldots, \lambda_{l-1}, \lambda_l=1$ in $[-1, 1] \subset \mathbb{R}$
satisfying \[h \approx_{\epsilon} \sum_{i=1}^l \lambda_i p_i,\qquad \sum_{i=1}^l p_i \neq 1.\]
By splitting the last term $p_l$ into two new projections if necessary,
we may assume that $[p_l]_0 = [p]_0$.
We put $G:=\{\alpha_s: s\in \mathbb{F}_\infty\}$ for short.
Take $q, r_1\in A^{\rm p} \setminus\{0, 1\}$
satisfying
\[q \perp r_1,\qquad q+r_1 \perp \sum_{i=1}^l p_i,\qquad [q]_0=[\sum_{i=1}^{l-1} p_i]_0.\]
Next we fix a real number
\[0<\delta< \min \left\{1,\ \frac{\epsilon-\|h-\sum_{i=1}^l \lambda_i p_i\|}{4l}\right\}.\]
By applying condition (1) to $(\sum_{i=1}^{l-1} p_i, p_l) \in {\rm P}(A; [q]_0, [p_l]_0)$,
we obtain $\gamma_1\in G$ satisfying
\[\sum_{i=1}^{l-1} \gamma_1(p_i) \approx_{\delta} q,\qquad \gamma_1(p_l) \approx_\delta p_l.\]
By Lemma 7.2.2 (1) in \cite{BO}, one can take  $u \in \mathcal{M}(A)^{\rm u}$
satisfying \[\|u-1\|< 4\delta,\qquad u\left(\sum_{i=1}^{l-1}\gamma_1(p_i)\right)u^\ast =q.\]
Set $q_i:= u \gamma_1(p_i) u^\ast (\approx_{8\delta} \gamma_1(p_i))$ for $i=1, \ldots, l-1$.
Then, since $|\lambda_i| \leq 1$ for all $i$, we have
\[\sum_{i=1}^l \lambda_i \gamma_1(p_i)\approx_{8l\delta} p_l + \sum_{i=1}^{l-1} \lambda_i q_i.\] 
Set
\[x_2:= \frac{1}{2}\sum_{i=1}^{l-1} \lambda_i(p_i +q_i) \in (1-p_l-r_1)A(1-p_l-r_1).\] 
Note that $x_2$ is self-adjoint and $\|x_2\|\leq 1/2$.
Since $4l\delta + \|h-\sum_{i=1}^l \lambda_i p_i\|<\epsilon$, we obtain
\[h_2:=\frac{1}{2}(h+ \gamma_1(h)) \approx_{\epsilon} p_l + x_2.\]
Next we apply the same argument to $p_l + x_2$ and $\epsilon - \| h_2-(p_l + x_2)\|$ instead of $h$ and $\epsilon$ (with the same $p_l$).
As a result we obtain $\gamma_2\in G$, $r_2\in A^{\rm p} \setminus \{0, 1\}$ with $r_2 \perp p_l$, and a  self-adjoint element $x_3$ in $(1-p_l-r_2)A(1-p_l-r_2)$ satisfying
\[h_3 := \frac{1}{2}[h_2+ \gamma_2(h_2)]\approx_\epsilon p_l + x_3,\qquad \|x_3\|\leq \frac{1}{2^2}.\]
Fix $N \in \mathbb{N}$ satisfying $2^{-(N-1)} <\epsilon$. 
By iterating this argument $N$ times,
we finally obtain $h_N \in X$ and $x_N \in (1- p_l)A(1-p_l)$
satisfying
\[h_N \approx_{\epsilon} p_l + x_N,\qquad \|x_N\|\leq \frac{1}{2^{N-1}}<\epsilon.\]
Choose $\gamma \in G$ satisfying $\gamma(p_l)\approx_{\epsilon} p$ (which exists by condition (1)).
We then obtain $p \approx_{3\epsilon} \gamma(h_N) \in X$.
Since $\epsilon>0$ is arbitrary, we conclude $p\in X$.
\end{proof}
\begin{Rem}\label{Rem:sf}
In the stably finite case, one cannot expect to find actions satisfying the conclusion of Proposition \ref{Prop:invsp}.
Indeed, let $A$ be a non-commutative \Cs-algebra.
Then, for any non-empty set $S$ of tracial states on $A$, the subspace $\bigcap_{\alpha\in \mathrm{Aut}(A)}\bigcap_{\tau \in S} \ker(\tau\circ \alpha) \subset A$
is proper, closed, self-adjoint, and invariant under $\mathrm{Aut}(A)$.
(Note that by the Hahn--Banach theorem, this subspace is nonzero.)
\end{Rem}

We now combine the subspace restriction obtained in Proposition \ref{Prop:invsp}
with condition (2) to effectively apply the Powers averaging argument \cite{Pow}, \cite{HS}.
As a result, we obtain strong restrictions of some reduced crossed product inclusions associated with \Cs-dynamical systems obtained in Proposition \ref{Prop:upert}.
\begin{Thm}\label{Thm:inter}
Let $\alpha \colon \mathbb{F}_\infty \curvearrowright A$
be an action on a purely infinite simple \Cs-algebra
satisfying conditions $(1)$ and $(2)$ in Proposition \ref{Prop:upert}.
Let $\beta \colon \mathbb{F}_\infty \curvearrowright B$ be an action on a simple \Cs-algebra with $B^\beta \neq 0$.
Then $0$, $B \rca{\beta} \mathbb{F}_\infty$, $(A\otimes B) \rca{\alpha \otimes \beta}\mathbb{F}_\infty$ are the only possible \Cs-subalgebras of $(A\otimes B) \rca{\alpha \otimes \beta}\mathbb{F}_\infty$ invariant under multiplications by
$B \rca{\beta} \mathbb{F}_\infty$.
\end{Thm}
Thus, when we additionally assume that $A$ is unital,
the reduced crossed product inclusion  $B \rca{\beta} \mathbb{F}_\infty \subset (A\otimes B) \rca{\alpha \otimes \beta}\mathbb{F}_\infty$ has no intermediate \Cs-algebras.
\begin{proof}
Let $C$ be a \Cs-subalgebra of $(A\otimes B) \rca{\alpha \otimes \beta} \mathbb{F}_\infty$ as in the statement.
We first consider the case that $E(C) \subset B$.
When $A$ is non-unital, this implies $C=0$.
When $A$ is unital, since $C=Cu_s$ for all $s\in \mathbb{F}_\infty$,
by Proposition 3.4 of \cite{Suz17},
we have $C\subset B \rca{\beta} \mathbb{F}_\infty$.
Since $B \rca{\beta} \mathbb{F}_\infty$ is simple (\cite{HS}, Theorem I), this yields $C=0$ or $C= B \rca{\beta} \mathbb{F}_\infty$.

We next consider the case $E(C)\not\subset B$.
Observe that $E(C)$ is an $(\alpha \otimes \beta)$-invariant self-adjoint subspace of $A\otimes B$.
Hence the elements of the form
$(\id_A \otimes \varphi)(E(c))$,
where $\varphi$ is a pure state on $B$ and $c \in C$,
span an $\alpha$-invariant self-adjoint subspace of $A$.
By (the easy part of) Theorem \ref{Thm:TS}, this subspace is not contained in $\mathbb{C}$.
Therefore, by Proposition \ref{Prop:invsp}, 
for any $\epsilon>0$ and any $p\in A^{\rm p}$, one can choose pure states $\varphi_1, \ldots, \varphi_n$ on $B$
and $c_1, \ldots, c_n \in C \setminus\{0\}$ satisfying
\[\sum_{i=1}^n(\id_A \otimes \varphi_i)(E(c_i))\approx_\epsilon p.\]
By the Akemann--Anderson--Pedersen excision theorem \cite{AAP} (\cite{BO}, Theorem 1.4.10)
(applied to each $\varphi_i$), for each $i=1, \ldots, n$,
one can choose $b_i \in B_+$ satisfying
\[\|b_i\|=1,\qquad
b_i E(c_i) b_i \approx_{\epsilon/2n} [(\id_A \otimes \varphi_i)(E(c_i))]\otimes b_i^2.\]
We fix an element $b\in (B^\beta)_{+}$ with $\|b\|=1$.
By applying Lemma \ref{Lem:simple} to each $b_i^2$ in $B$, we obtain finite sequences $(v_{i, j})_{j=1}^{l(i)}$, $i=1, \ldots, n$, in $B$
satisfying
\[\|\sum_{j=1}^{l(i)} v_{i, j} b_i^2 v_{i, j}^\ast - b\|<\frac{\epsilon}{2n\|c_i\|},\qquad \sum_{j=1}^{l(i)} v_{i, j} v_{i, j}^\ast \leq 1.\]
Set $x_{i, j}:= v_{i, j}b_i \in B \subset \mathcal{M}((A\otimes B) \rca{\alpha \otimes \beta}\mathbb{F}_\infty)$ for $i=1, \ldots, n$ and $j=1, \ldots, l(i)$.
We then obtain
\[ E(\sum_{i=1}^n\sum_{j=1}^{l(i)} x_{i, j} c_i x_{i, j}^\ast)=\sum_{i=1}^n\sum_{j=1}^{l(i)} x_{i, j} E(c_i) x_{i, j}^\ast \approx_{\epsilon} \left[\sum_{i=1}^n (\id_A \otimes \varphi_i)(E(c_i))\right]\otimes b \approx_{\epsilon}  p \otimes b.\]
Summarizing the result, we have shown that, for any $p\in A^{\rm p}$, any $b\in (B^\beta)_+$, and any $\epsilon>0$, there exists $c\in C$ satisfying
\[E(c) \approx_\epsilon p \otimes b.\]

Fix $p\in A^{\rm p}$, $\epsilon>0$, $b\in (B^\beta)_+\setminus \{0\}$,
and take $c\in C$ satisfying $E(c) \approx_\epsilon p \otimes b$.
We further assume that the stabilizer subgroup of $p$
contains at least two canonical generating elements $s_1$, $s_2$ of $\mathbb{F}_\infty$.
Choose $c_0 \in (A \otimes B) \ac \mathbb{F}_\infty$ satisfying
$c_0 \approx_{\epsilon} c$, $E(c_0)=p \otimes b$.
We apply the Powers argument \cite{Pow}, \cite{HS} to $c_0 - p\otimes b$ by using $s_1, s_2$ (cf.~  \cite{Suzmin}, Lemma 3.8).
As a result, we obtain a sequence $g_1, \ldots, g_m$ in $\langle s_1, s_2 \rangle$ (the subgroup generated by $s_1$ and $s_2$)
satisfying
\[\frac{1}{m} \sum_{i=1}^m u_{g_i} c_0 u_{g_i}^\ast \approx_{\epsilon} p \otimes b.\]
This implies 
\[ p \otimes b\approx_{2\epsilon}\frac{1}{m} \sum_{i=1}^m u_{g_i} c u_{g_i}^\ast \in C.\]
Since $\epsilon>0$ is arbitrary,
we conclude $p \otimes b\in C$.
Since $\alpha$ satisfies condition (2), $A$ is of real rank zero \cite{Zha}, and $B$ is simple, we obtain
$C=(A \otimes B) \rca{\alpha\otimes \beta}\mathbb{F}_\infty$.
\end{proof}

\section{Rigidity properties of inclusions}\label{Sec:rigidity}
We establish a rigidity of automorphisms for inclusions obtained in Theorem \ref{Thm:inter}.
It is notable that the same property plays an important role in the proof of the Galois correspondence theorem of Izumi--Longo--Popa \cite{ILP}
(see also \cite{Lon}).
We then slightly modify Proposition \ref{Prop:upert} to give rigid inclusions.
This completes the proof of the Main Theorem.

By using the spectral $M$-subspaces (see \cite{Ped}, Definition 8.1.3),
we obtain the following result from Proposition \ref{Prop:invsp}.
\begin{Lem}\label{Lem:comm}
Let $\alpha \colon \mathbb{F}_\infty \curvearrowright A$ be a \Cs-dynamical
system as in Proposition \ref{Prop:invsp}.
Assume that $A$ is unital.
Then there is no non-trivial automorphism of $A$
commuting with $\alpha$.
\end{Lem}
\begin{proof}
Let $\beta$ be an automorphism of $A$ commuting with $\alpha$.
Then for any closed subset $K \subset \widehat{\mathbb{Z}}$ with $K^{-1}=K$,
the $M$-subspace
$M^\beta(K)$ forms a closed self-adjoint $\alpha$-invariant subspace of $A$.
Since $1\in M^\beta(\{1\})$, by Theorem 8.1.4 (iv), (ix) of \cite{Ped}, we have $1\not\in M^\beta(K)$ whenever $1\not\in K$.
By Proposition \ref{Prop:invsp}, this forces that $M^\beta(K)=0$ for
any compact subset $K\subset  \widehat{\mathbb{Z}} \setminus \{1\}$.
Consequently, by Theorem 8.1.4 (iii), (vii), (viii), (ix) of \cite{Ped},
we have $M^\beta(\{1\})=A$.
Corollary 8.1.8 in \cite{Ped} now yields $\beta=\id_A$.
\end{proof}
For a completely positive map $\Phi \colon A \rightarrow B$
between \Cs-algebras,
when it extends to a completely positive map $\mathcal{M}(A) \rightarrow \mathcal{M}(B)$ which is strictly continuous on the unit ball,
we denote by $\Phi^{\mathcal{M}}$ such a (unique) extension.
Such an extension exists if $\Phi$ maps
an approximate unit of $A$ to that of $B$; see Corollary 5.7 in \cite{Lan}.
It is obvious that all completely positive maps appearing below
satisfy this condition.
\begin{Prop}\label{Prop:rigid}
Let $\alpha, \beta$ be as in Theorem \ref{Thm:inter}.
Assume that $A$ is unital.
Then the inclusion $C:=B \rca{\beta} \mathbb{F}_\infty \subset (A\otimes B) \rca{\alpha \otimes \beta}\mathbb{F}_\infty:=D$ has the following property.
If two automorphisms $\gamma_1$, $\gamma_2$ of $D$ coincide on $C$,
then $\gamma_1=\gamma_2$.
\end{Prop}
\begin{proof}
To prove the statement, it suffices to show the following claim.
Any automorphism $\gamma$ on $D$ with $\gamma|_{C} =\id _{C}$ must be trivial.
Given such $\gamma$.
We will show that $\gamma^{\mathcal{M}}(A)=A$.
To see this, we first show that $E^{\mathcal{M}}(\gamma^{\mathcal{M}}(A)) \subset A$.
Take $x \in E^{\mathcal{M}}(\gamma^{\mathcal{M}}(A)) \subset \mathcal{M}(A\otimes B)$.
Then note that $x$ commutes with $B$ (by standard arguments on multiplicative domains).
Therefore, for any state $\varphi$ on $A$,
we have \[(\varphi \otimes \id_B)^{\mathcal{M}}(x)\in Z(\mathcal{M}(B)) = \mathbb{C}.\]
Here and below, we regard $\mathcal{M}(A)=A$ and $\mathcal{M}(B)$
as \Cs-subalgebras of $\mathcal{M}(A\otimes B)$ in the obvious way.
This shows that, for any state $\psi$ on $B$,
\[(\varphi \otimes \id_B)^{\mathcal{M}}((\id_A \otimes \psi)^{\mathcal{M}}(x))=
(\id_A \otimes \psi)^{\mathcal{M}}((\varphi \otimes \id_B)^{\mathcal{M}}(x))=(\varphi \otimes \id_B)^{\mathcal{M}}(x).\]
Observe that the maps $(\varphi \otimes \id_B)^{\mathcal{M}}$; $\varphi$ a state on $A$,
separate the points of $\mathcal{M}(A\otimes B)$.
Therefore we conclude
\[x= (\id_A \otimes \psi)^{\mathcal{M}}(x)\in A.\]
Since $\gamma^{\mathcal{M}}(u_s)=u_s$ for all $s\in \mathbb{F}_\infty$,
the unital completely positive map
\[(E^{\mathcal{M}} \circ \gamma^{\mathcal{M}})|_A \colon A \rightarrow A\] is $\mathbb{F}_\infty$-equivariant.
Proposition 3.1 shows that $E^{\mathcal{M}}(\gamma^{\mathcal{M}}(A))$ is dense in $A$.

Fix $b\in (B^\beta)_+\setminus \{0\}$ with $\|b\|=1$ and $\epsilon>0$.
Let $p \in A^{\rm p}$ be a projection whose stabilizer subgroup of $\alpha$ contains
at least two canonical generating elements $s_1$, $s_2$ of $\mathbb{F}_\infty$.
Choose $x\in A$ satisfying $E^{\mathcal{M}}(\gamma^{\mathcal{M}}(x))\approx_\epsilon p$.
By applying the Powers argument \cite{Pow}, \cite{HS} to $\gamma^{\mathcal{M}}(x) b=\gamma(x\otimes b) \in D$ by using $s_1$ and $s_2$ (cf.~the proof of Theorem \ref{Thm:inter} or \cite{Suzmin}, Lemma 3.8),
we obtain a sequence $g_1, \ldots, g_n$ in $\langle s_1, s_2 \rangle$
satisfying
\[p\otimes b \approx_\epsilon \frac{1}{n}\sum_{i=1}^n u_{g_i} \gamma(x \otimes b) u_{g_i}^\ast
=\frac{1}{n}\sum_{i=1}^n \gamma(\alpha_{g_i}(x) \otimes b).\]
Since $\epsilon>0$ is arbitrary,
we obtain $p \otimes b \in  \gamma(A \otimes b)$.
(Note that $\gamma$ is isometric hence $\gamma(A \otimes b)$ is closed in $D$.)
By condition (2) of $\alpha$ and \cite{Zha}, we obtain
\[A \otimes b \subset \gamma(A \otimes b)=\gamma^{\mathcal{M}}(A)b.\]
By Lemma \ref{Lem:simple},
one can choose a net $((v_{i, \lambda})_{i=1}^{n(\lambda)})_{\lambda \in \Lambda}$
of finite sequences in $B$ satisfying
\[\sum_{i=1}^{n(\lambda)}v_{i, \lambda} v_{i, \lambda}^\ast \leq 1 \qquad{\rm for~ all~ }\lambda\in \Lambda,\]

 \[\lim_{\lambda \in \Lambda} \sum_{i=1}^{n(\lambda)}v_{i, \lambda} b v_{i, \lambda}^\ast=1 \qquad {\rm in~ the~ strict~ topology~ of~} \mathcal{M}(D).\]

Now for any $a\in A$, choose $x\in A$ with
$a\otimes b=\gamma^{\mathcal{M}}(x) b$.
Then, for any $\lambda \in \Lambda$,
\[a\otimes \left(\sum_{i=1}^{n(\lambda)}v_{i, \lambda} b v_{i, \lambda}^\ast \right)= \gamma^{\mathcal{M}}(x) \left(\sum_{i=1}^{n(\lambda)}v_{i, \lambda} b v_{i, \lambda}^\ast \right).\]
By letting $\lambda$ tend to infinity,
we obtain $a=\gamma^{\mathcal{M}}(x)$.
Applying the same argument to $\gamma^{-1}$, we obtain $A=\gamma^{\mathcal{M}}(A)$.
Thus $\gamma^{\mathcal{M}}|_A$ defines an automorphism on $A$.
Since $\gamma^{\mathcal{M}}(u_s)=u_s$ for all $s\in \mathbb{F}_\infty$,
the automorphism $\gamma^{\mathcal{M}}|_A$ commutes with $\alpha$.
 Lemma \ref{Lem:comm} therefore implies $\gamma^{\mathcal{M}}|_A=\id_A$.
Since $A \cdot C$ generates $D$,
we conclude $\gamma=\id_{D}$.
\end{proof}
We now construct rigid inclusions.
We first recall the definition.
\begin{Def}[\cite{Ham79b}, Definition 2.4]\label{Def:rigid}
An inclusion $A \subset B$ of \Cs-algebras
is said to be \emph{rigid} if
the identity map $\id_B$ is the only completely positive map
$\Phi\colon B \rightarrow B$ satisfying $\Phi|_A=\id_A$.
\end{Def}
By slightly modifying Proposition \ref{Prop:upert} (under a stronger assumption),
we obtain \Cs-dynamical systems satisfying a stronger condition which is useful to study the rigidity
of associated inclusions.

For a separable \Cs-algebra $A$,
when we equip
the automorphism group $\mathrm{Aut}(A)$ of $A$ with the point-norm topology,
it forms a Polish group.
(The point-norm topology of $\mathrm{Aut}(A)$ is the weakest topology on $\mathrm{Aut}(A)$
making the evaluation maps $\alpha \mapsto \alpha(a) \in A$ norm continuous for all $a\in A$.)
Indeed, take a dense sequence $(a_n)_{n=1}^\infty$ in the unit ball of $A$.
Then it is not hard to see that the metric $d$ on $\mathrm{Aut}(A)$ given by
\[d(\alpha, \beta):= \sum_{n=1}^\infty \frac{1}{2^{n}}\left(\|\alpha(a_n)-\beta(a_n)\|+\|\alpha^{-1}(a_n)-\beta^{-1}(a_n)\|\right); \qquad \alpha, \beta \in\mathrm{Aut}(A),\]
 confirms the statement.
Denote by $\overline{{\rm Inn}}(A)$ the closure of the inner automorphism group ${\rm Inn}(A)$ in $\mathrm{Aut}(A)$.
We say that a \Cs-dynamical system $\alpha \colon \Gamma \curvearrowright A$ is \emph{pointwise approximately inner}
if $\alpha_s \in \overline{{\rm Inn}}(A)$ for all $s\in \Gamma$.
\begin{Prop}\label{Prop:upert2}
Let $\alpha \colon \mathbb{F}_\infty \curvearrowright A$ be a pointwise  approximately inner \Cs-dynamical system on
a separable purely infinite simple \Cs-algebra $A$.
Then there exists an inner perturbation
$\beta$ of $\alpha$ satisfying the following conditions.
\begin{enumerate}[\upshape(1)]
\item The set $\{\beta_s:s\in \mathbb{F}_\infty\}$ is dense in $\overline{{\rm Inn}}(A)$.
\item Let $S_\beta$ denote the set of all $p \in A^{\rm p}$ whose stabilizer subgroup of $\beta$ contains at least two
canonical generating elements of $\mathbb{F}_\infty$.
Then $S_\beta$ is dense in $A^{\rm p}$ in norm.
\end{enumerate}
\end{Prop}
\begin{proof}
We split the canonical generating set $S$ of $\mathbb{F}_\infty$ into two infinite subsets: $S= S_1 \sqcup S_2$.
We first perturb $\alpha_s$; $s\in S_2$ by inner automorphisms as in the proof of Proposition \ref{Prop:upert}
to ensure condition (2).
We next choose a dense sequence $(\gamma_n)_{n=1}^\infty$ in $\overline{{\rm Inn}}(A)$.
Fix a bijective map $f\colon  \mathbb{N}\times \mathbb{N}\rightarrow S_1$.
Since each $\alpha_s$ is approximately inner,
there exist $v_s\in \mathcal{M}(A)^{\rm u}$; $s \in S_1$,
satisfying \[\lim_{m\rightarrow \infty} \ad(v_{f({n, m})})\circ\alpha_{f(n, m)}=\gamma_n
\qquad {\rm for~all~}n\in \mathbb{N}.\]
These unitary elements define the desired inner perturbation of $\alpha$.
\end{proof}
\begin{Cor}\label{Cor:ameO}
There is an amenable action of  $\mathbb{F}_\infty$ on the Cuntz algebra $\mathcal{O}_\infty$ satisfying conditions $(1)$ and $(2)$ in Proposition \ref{Prop:upert2}.
\end{Cor}
\begin{proof}
Recall from the proof of Theorem 5.1 of \cite{Suz19}
that $\mathbb{F}_\infty$ admits an amenable action $\alpha$ on $\mathcal{O}_\infty$.
It follows from the construction that $\alpha$ is pointwise approximately inner.
Now applying Proposition \ref{Prop:upert2} (and Lemma \ref{Lem:ame}) to $\alpha$, we obtain the desired action.
\end{proof}
\begin{Rem}
It follows from Lemma \ref{Lem:trans} that condition (1) of Proposition \ref{Prop:upert2}
is stronger than condition (1) of Proposition \ref{Prop:upert}.
\end{Rem}
\begin{Lem}\label{Lem:ucp}
Let $\alpha\colon \mathbb{F}_\infty \curvearrowright A$ be an action
on a unital purely infinite simple \Cs-algebra satisfying condition $(1)$ in Proposition \ref{Prop:upert2}.
Then there is no $\mathbb{F}_\infty$-equivariant unital completely positive map $\Phi \colon A \rightarrow A$ other than $\id_A$ $($that is, $\mathbb{C}\subset A$ is \emph{$\mathbb{F}_\infty$-rigid}$)$.
\end{Lem}
\begin{proof}
Let $\Phi$ be as in the statement.
By the assumption on $\alpha$, all inner automorphisms are in
the closure of $\{\alpha_s: s\in \mathbb{F}_\infty\}$ in $\mathrm{Aut}(A)$. 
Therefore, for any $u\in A^{\rm u}$ and any $x\in A$,
we have $\Phi(uxu^\ast)=u\Phi(x)u^\ast$.
Applying the equality to $x=p \in A^{\rm p}\setminus\{0, 1\}$ and unitary elements
$u$ in $pAp \oplus (1-p)A(1-p) \subset A$,
we obtain
$\Phi(p)=u\Phi(p)u^\ast$.
Thus $\Phi(p)$ commutes with $pAp \oplus (1-p)A(1-p)$.
Note that since $A$ is simple, so are $pAp$ and $(1-p)A(1-p)$.
We therefore obtain
\[\Phi(p)  = \lambda_1 p + \lambda_2 (1-p) \qquad  {\rm for~ some~} \lambda_1, \lambda_2 \geq 0.\]
We will show that $\lambda_2=0$.
Take $v\in A^{\rm u}$ which satisfies $q:=vpv^\ast \lneq 1- p$ (see \cite{Cun}).
Then \[\Phi(q)= v\Phi(p) v^\ast=\lambda_1 q+ \lambda_2(1-q).\]
This yields
\[\Phi(p+q)= \lambda_1 (p+q) + \lambda_2(1-p-q)+\lambda_2= (\lambda_1+\lambda_2)(p+q)+2\lambda_2(1-p-q).\]
By iterating this argument, for any $N\in \mathbb{N}$,
one can find $r_N\in A^{\rm p}\setminus\{0, 1\}$
satisfying
\[\Phi(r_N) =[\lambda_{1}+ (2^{N}-1)\lambda_2] r_N + 2^N\lambda_{2}(1-r_N).\]
Since $\Phi$ is contractive, this forces $\lambda_2 =0$.
Thus $\Phi(p)\leq p$ for all $p\in A^{\rm p}$.
Since $\Phi$ is unital, these inequalities imply
$\Phi(p)=p$ for all $p\in A^{\rm p}$.
Since $A^{\rm p}$ spans a dense subspace of $A$ \cite{Zha},
we conclude $\Phi=\id_A$.
\end{proof}
\begin{Thm}\label{Thm:rigid}
Let $\alpha\colon \mathbb{F}_\infty \curvearrowright A$ be a \Cs-dynamical system
on a unital purely infinite simple \Cs-algebra satisfying condition $(1)$ in Proposition \ref{Prop:upert2}.
Let $\beta \colon \mathbb{F}_\infty \curvearrowright B$ be  a \Cs-dynamical system on a simple \Cs-algebra.
Then the inclusion \[C:=B \rca{\beta} \mathbb{F}_\infty \subset (A\otimes B) \rca{\alpha \otimes \beta}\mathbb{F}_\infty:=D\] is rigid.
\end{Thm}
\begin{proof}
Let $\Phi \colon D \rightarrow D$
be a completely positive map
with $\Phi|_{C}=\id_C$.
Observe that $C$ contains
an approximate unit of $D$.
Hence by Corollary 5.7 of \cite{Lan}, the $\Phi$ has a strictly continuous extension
$\Phi^{\mathcal{M}} \colon \mathcal{M}(D) \rightarrow \mathcal{M}(D)$.
Since $\Phi|_B=\id_B$, standard arguments on multiplicative domains show that
\[(E^\mathcal{M} \circ \Phi^{\mathcal{M}})(A)\subset \mathcal{M}(A\otimes B) \cap B'=A.\] (For the proof of the last equality, see the proof of Proposition \ref{Prop:rigid}.) 
Since $\Phi^{\mathcal{M}}(u_s)=u_s$ for all $s\in \mathbb{F}_\infty$,
the unital completely positive map \[(E^\mathcal{M} \circ \Phi^{\mathcal{M}})|_A \colon A \rightarrow A\] is $\mathbb{F}_\infty$-equivariant.
Therefore Lemma \ref{Lem:ucp} implies $(E^\mathcal{M} \circ \Phi^{\mathcal{M}})|_A=\id_A$ .
Observe that for any $u\in A^{\rm u}$ and any $x\in \mathcal{M}(D)$
satisfying $E^{\mathcal{M}}(x)=0$, we have \[\|u+x\|^2\geq \|E^{\mathcal{M}}((u+x)^\ast (u+x))\|=1+ \|E^\mathcal{M}(x^\ast x)\|.\]
Since $E^\mathcal{M} \colon \mathcal{M}(D) \rightarrow \mathcal{M}(A\otimes B)$ is faithful, we obtain $\|u+x\|>\|u\|$ unless $x=0$.
As both $E^\mathcal{M}$ and $\Phi^{\mathcal{M}}$ are contractive,
the equality $(E^\mathcal{M} \circ \Phi^{\mathcal{M}})|_{A^{\rm u}}=\id_{A^{\rm u}}$ implies that $\Phi^{\mathcal{M}}|_{A^{\rm u}}=\id_{A^{\rm u}}$.
Since $A^{\rm u}\cdot C$ spans a dense subspace of 
$D$,
we conclude $\Phi=\id_D$.
\end{proof}
Now by combining Proposition \ref{Prop:upert2} and Theorems \ref{Thm:inter}, \ref{Thm:rigid}, we obtain the Main Theorem.

Before closing this section, we record the following elementary lemma on rigidity of \Cs-algebra inclusions. This lemma will be used in the next section.
\begin{Lem}\label{Lem:rigidcorner}
Let $A \subset B$ be a rigid inclusion of unital purely infinite simple \Cs-algebras.
Let $p\in A^{\rm p}$.
Then the inclusion $pAp \subset pBp$ is also rigid.
\end{Lem}
\begin{proof}
Assume that  the inclusion $pAp \subset pBp$ is not rigid.
Take a completely positive map
$\Phi \colon pBp \rightarrow pBp$
satisfying $\Phi|_{pAp}=\id_{pAp}$ and $\Phi\neq \id_{pBp}$.
Choose $v\in A$ satisfying $v^\ast v=1$, $vv^\ast \leq p$.
Define $\Psi\colon B \rightarrow B$
to be $\Psi(x):= v^\ast\Phi(v x v^\ast)v$, $x\in B$.
Then for any $x\in pBp$, since $vp \in pAp$,
we obtain
\[\Psi(x) = v^\ast\Phi(vpxpv^\ast )v=v^\ast vp\Phi(x)pv^\ast v=\Phi(x).\]
In particular, $\Psi\neq \id_B$.
Also, for any $a\in A$, as $v a v^\ast \in pAp$, we have
\[\Psi(a)= v^\ast \Phi(va v^\ast)v = v^\ast v a v^\ast v =a.\]
In summary, we obtain $\Psi|_A=\id_A$, $\Psi\neq \id_B$.
Thus the inclusion $A\subset B$ is not rigid.
\end{proof}
\section{Applications to Kirchberg algebras: proofs of Theorems \ref{Thmint:Main} to \ref{Thmint:3}}\label{Sec:Kir}
We now apply the Main Theorem to obtain the main results.
\begin{proof}[Proof of Theorem \ref{Thmint:Main}]
Let $\beta \colon \mathbb{F}_\infty \curvearrowright \mathcal{O}_\infty$
be an action obtained in Corollary \ref{Cor:ameO}.
We show that $B:=A \rca{\alpha} \mathbb{F}_\infty \subset C:=(A\otimes \mathcal{O}_\infty) \rca{\alpha\otimes \beta}\mathbb{F}_\infty$ gives the desired ambient \Cs-algebra.

We first show that $C$ is a Kirchberg algebra.
Clearly $C$ is separable.
Since $A \otimes \mathcal{O}_\infty$ is simple, purely infinite,
and $\alpha\otimes \beta$ is outer (because of its amenability
and the fact that $\mathbb{F}_\infty$ has no non-trivial amenable normal subgroup),
it follows from Kishimoto's theorem \cite{Kis} that $C$ is purely infinite and simple (see e.g.~Lemma 6.3 of \cite{Suz19b} for details).
Since $A\otimes \mathcal{O}_\infty$ is nuclear, so is $C$ by the amenability of $\alpha \otimes \beta$. Thus $C$ is a Kirchberg algebra.

By Theorem \ref{Thm:inter}, the inclusion indeed has no intermediate \Cs-algebras.
By Theorem \ref{Thm:rigid}, the inclusion is rigid.
Since the inclusion $\mathbb{C}\subset \mathcal{O}_\infty$ is a KK-equivalence \cite{Cun}, \cite{PV80},
so is $A \subset A \otimes \mathcal{O}_\infty$.
Now it follows from Theorem 16 of \cite{Pim} (see also \cite{PV})
that the inclusion $B\subset C$ is a KK-equivalence.
(Proof: We apply the exact sequences in Theorem 16 of \cite{Pim}
to a fixed free action of $\mathbb{F}_\infty$ on a countable tree.
Observe that for any countable set $I$,
the inclusion $\bigoplus_I A \subset  \bigoplus_I (A \otimes \mathcal{O}_\infty)$ is
a KK-equivalence.
By the Five Lemma and naturality of the exact sequences, the inclusion map $\iota \colon B \rightarrow C$ induces group isomorphisms
\[\varphi\colon \mathrm{KK}(C, B) \rightarrow \mathrm{KK}(B, B),\qquad
\psi \colon  \mathrm{KK}(C, B) \rightarrow \mathrm{KK}(C, C).\]
Put $x:= \varphi^{-1}(1_B)$, $y:=\psi^{-1}(1_C)$.
It then follows from the definition that
\[[\iota]\hat{\otimes}_{C} x =\varphi(x)=1_B,\qquad y\hat{\otimes}_{B}[\iota]= \psi(y)=1_C.\]
Thus $x=y$ and $\iota$ is a KK-equivalence.)
\end{proof} 
\begin{Rem}\label{Rem:gengr}
Recall that any discrete exact group $\Gamma$ admits an amenable
action on a unital purely infinite simple nuclear \Cs-algebra of density character $\sharp \Gamma$; see the proof of Proposition B in \cite{Suzeq}.
For the existence of a nuclear minimal ambient \Cs-algebra, our construction works for groups
of the form $\mathbb{F}_\Lambda \ast \Lambda$ for any infinite group $\Lambda$ with the approximation property \cite{HK}.
(However the resulting ambient algebras would be mysterious, cf.~ \cite{Suz19b}).
In particular, the reduced group \Cs-algebras of uncountable free groups
admit a nuclear minimal ambient \Cs-algebra.
\end{Rem}

\begin{proof}[Proof of Theorem \ref{Thmint:Kir}] 
Let $A$ be a Kirchberg algebra.
We have constructed, in the proof of the Proposition in \cite{Suzfp} (see also the proof of Theorem 5.1 in \cite{Suz19}), an action $\alpha \colon \mathbb{F}_\infty \curvearrowright C$
on a unital Kirchberg algebra $C$ in the bootstrap class whose reduced crossed product $C \rca{\alpha}\mathbb{F}_\infty$ is non-nuclear, purely infinite simple, and KK-equivalent to $\mathcal{O}_\infty$.
Let $\beta \colon \mathbb{F}_\infty \curvearrowright \mathcal{O}_\infty$
be an amenable action obtained in Corollary \ref{Cor:ameO}.
As shown in the proofs of the Proposition in \cite{Suzfp} and Theorem 5.1 in \cite{Suz19} (by using \cite{Kir}, \cite{Phi}),
the crossed product $(C \otimes \mathcal{O}_\infty) \rca{\alpha\otimes \beta}\mathbb{F}_\infty$
is stably isomorphic to $\mathcal{O}_\infty$.
Fix a projection \[p \in C \rca{\alpha}\mathbb{F}_\infty\] which generates
K$_0((C \otimes \mathcal{O}_\infty) \rca{\alpha\otimes \beta}\mathbb{F}_\infty) \cong \mathbb{Z}$.
(This is possible by \cite{PV} and \cite{Cun}.)
Denote by $1$ the trivial action of $\mathbb{F}_\infty$ on $A$.
Then by the Kirchberg $\mathcal{O}_\infty$-absorption theorem \cite{KP},
the corner $p[(A\otimes C \otimes \mathcal{O}_\infty) \rca{1 \otimes \alpha\otimes \beta}\mathbb{F}_\infty]p$ is isomorphic to $A$.
The desired subalgebra of $A \cong p[(A\otimes C \otimes \mathcal{O}_\infty) \rca{1 \otimes \alpha\otimes \beta}\mathbb{F}_\infty]p$ is given by
\[p[(A\otimes C) \rca{1 \otimes \alpha}\mathbb{F}_\infty]p \subset p[(A\otimes C \otimes \mathcal{O}_\infty) \rca{1 \otimes \alpha\otimes \beta}\mathbb{F}_\infty]p.\]
Indeed, by Theorem \ref{Thm:inter} and \cite{Suz19}, Lemma 5.2,
the inclusion has no intermediate \Cs-algebras.
By Theorem \ref{Thm:rigid} and Lemma \ref{Lem:rigidcorner}, the inclusion is rigid.
Note that the corner $p[(A\otimes C) \rca{1 \otimes \alpha}\mathbb{F}_\infty]p$
is isomorphic to $A\otimes p(C\rca{\alpha}\mathbb{F}_\infty)p$,
which is not nuclear by the choice of $\alpha$.
By \cite{Pim} or \cite{PV} (see the proof of Theorem \ref{Thmint:Main} for details),
the inclusion gives a KK-equivalence.

We next construct an ambient non-exact \Cs-algebra of
$p[(A \otimes C \otimes \mathcal{O}_\infty) \rca{1\otimes \alpha\otimes \beta}\mathbb{F}_\infty]p \cong A$
as in the statement.
We first take a non-exact unital simple separable \Cs-algebra $D_0$ such that
the inclusion $\mathbb{C}\subset D_0$ is a KK-equivalence.
(Example: Take a unital non-exact separable \Cs-algebra $P_0$.
Set $P:=\{f\in C([0, 1]^2, P_0):f(t, 0)\in \mathbb{C} {\rm~for~all~}t\in [0, 1]\}$.
Note that $P$ is non-exact and homotopy equivalent to $\mathbb{C}$.
Take a faithful state $\varphi$ on $P$ satisfying the conditions in Theorem 2 of \cite{Dy}.
By Exercise 4.8.1 in \cite{BO}, there is a Hilbert $P$-bimodule
 whose Toeplitz--Pimsner algebra $D_0$ \cite{Pim}
is isomorphic to the reduced free product $(P, \varphi) \ast (\mathcal{T}, \omega)$.
Here $\mathcal{T}$ is the Toeplitz algebra and $\omega$ is a non-degenerate state on $\mathcal{T}$.
By Theorem 4.4 of \cite{Pim2}, the inclusion $P \subset D_0$ is a KK-equivalence.
By Theorem 2 of \cite{Dy}, $D_0$ is simple.
Thus $D_0$ gives the desired \Cs-algebra.)
Set $D:= D_0 \otimes \mathcal{O}_\infty$.
Then $D$ is separable, purely infinite, simple, and the inclusion
$\mathbb{C} \subset D$ gives a KK-equivalence.
By applying Proposition \ref{Prop:upert2} to the trivial action $\mathbb{F}_\infty \curvearrowright D$,
we obtain an (inner) action $\gamma \colon \mathbb{F}_\infty \curvearrowright D$
satisfying conditions (1), (2) in Proposition \ref{Prop:upert2}.
By the same reasons as in the previous paragraph,
the inclusion
\[p[(A\otimes C \otimes \mathcal{O}_\infty) \rca{1\otimes \alpha\otimes \beta}\mathbb{F}_\infty]p \subset p[(A\otimes C \otimes \mathcal{O}_\infty \otimes D) \rca{1\otimes \alpha\otimes \beta\otimes \gamma}\mathbb{F}_\infty]p\]
is rigid, gives a KK-equivalence, and has no intermediate \Cs-algebras.
The non-exactness of the largest \Cs-algebra is obvious.
Finally, by Kirchberg's theorem (\cite{Ror}, Theorem 4.1.10 (i)), all these \Cs-algebras are purely infinite.
\end{proof}
\begin{Rem}
By a similar method to the Proposition in \cite{Suzfp} (by using \cite{Oz}), one can
arrange the smallest algebra in Theorem \ref{Thmint:Kir}
not having the completely bounded approximation property (see Section 12.3 of \cite{BO} for the definition).
\end{Rem}
\begin{proof}[Proof of Theorem \ref{Thmint:3}]

Recall that in the proof of \cite{Suz19}, Theorem 5.1,
we obtained an amenable action $\alpha \colon \mathbb{F}_\infty \curvearrowright D$ on a unital Kirchberg algebra and a projection $p\in D$ 
such that $p(D \rca{\alpha} \mathbb{F}_\infty)p$ is isomorphic to $\mathcal{O}_\infty$.
Denote by $1\colon \mathbb{F}_\infty \curvearrowright A$ the trivial action on $A$.
By the Kirchberg $\mathcal{O}_\infty$-absorption theorem \cite{KP},
$p((A\otimes D) \rca{1\otimes \alpha} \mathbb{F}_\infty)p \cong A$.

Let $B$ be a given unital separable \Cs-algebra.
Choose a faithful state $\varphi$ on $B$.
Let $\psi$ denote the state on $C([0, 1])$ defined by
the Riemann integral.
Then by Theorem 2 of \cite{Dy},
the reduced free product
\[P_0:=(B, \varphi) \ast (C[0, 1], \psi)\]
 is simple.
Note that by Theorem 4.8.5 of \cite{BO},
the canonical inclusion $B \subset P_0$ admits a faithful conditional expectation (as $\psi$ is faithful).
Set
\[P:=P_0 \otimes \mathcal{O}_\infty.\]
Then $P$ is unital, simple, separable, and purely infinite.
The canonical inclusion $B\subset P$ still admits a faithful conditional expectation.
Applying Proposition \ref{Prop:upert2}
to the trivial action of $\mathbb{F}_\infty$ on $P$,
we obtain an (inner) action $\beta\colon \mathbb{F}_\infty \curvearrowright P$
satisfying conditions (1) and (2) in the statement.
Now define \[C:= p[(A \otimes D \otimes P)\rca{1\otimes \alpha \otimes \beta} \mathbb{F}_\infty]p.\]
Observe that the map $x\in P \mapsto xp \in C$ defines a \Cs-algebra embedding.
We identify $B\subset P$ with \Cs-subalgebras of $C$ via this embedding.

We now show that $B\subset C$ admits a faithful conditional expectation.
Since $p \in D$, the canonical conditional expectation 
\[E\colon (A \otimes D \otimes P)\rca{1\otimes \alpha \otimes \beta} \mathbb{F}_\infty \rightarrow A\otimes D \otimes P\]
restricts to the faithful conditional expectation $\Phi \colon C \rightarrow p(A\otimes D\otimes P)p$.
Any faithful state $\omega$ on $p(A\otimes D)p$
induces a faithful conditional expectation
$\Psi \colon p(A \otimes D\otimes P)p \rightarrow P$
by the formula $\Psi(p(x \otimes y)p)=\omega(pxp)y$; $x\in A\otimes D$, $y\in P$.
The composite $\Psi\circ \Phi  \colon C \rightarrow P$
gives a faithful conditional expectation.
Since $B\subset P$ has a faithful conditional expectation,
consequently so does $B \subset C$.

It follows from Theorem \ref{Thm:inter} and \cite{Suz19}, Lemma 5.2 that
the inclusion
\[A \cong p((A\otimes D) \rca{1\otimes \alpha} \mathbb{F}_\infty)p \subset C\] has no intermediate \Cs-algebras.
By Theorem \ref{Thm:rigid} and Lemma \ref{Lem:rigidcorner},
the inclusion $A\subset C$ is rigid.
\end{proof}
\begin{Rem}\label{Rem:LF}
By a similar method to the proof of Theorem \ref{Thmint:3} (using \cite{Dy94} instead of \cite{Dy}),
one can confirm the following property for the free group factor $L(\mathbb{F}_\infty)$:
Any von Neumann algebra $M$ with separable predual
embeds into a factor $N$ with a normal faithful conditional expectation
which contains $L(\mathbb{F}_\infty)$ as a rigid maximal von Neumann subalgebra.
Here we say that a von Neumann subalgebra $M\subset N$ is \emph{rigid}
if $\id_M$ is the only normal completely positive map $\Phi \colon N \rightarrow N$
satisfying $\varphi|_M=\id_M$ 
\end{Rem}

\appendix
\section{Tensor splitting theorem for non-unital simple \Cs-algebras}
Here we record a few necessary and useful technical lemmas on non-unital \Cs-algebras.
Although these results would be known for some experts, we do not know an appropriate reference.
As a result of these lemmas, we obtain the tensor splitting theorem (cf.~\cite{GK}, \cite{Zac}, \cite{Zsi})
for non-unital simple \Cs-algebras.

An element of a \Cs-algebra $A$ is said to be {\it full}
if it generates $A$ as a closed ideal of $A$.

\begin{Lem}\label{Lem:full}
Let $A$ be a \Cs-algebra. Let $a$ be a full positive element of $A$.
Then for any finite subset $F$ of $A$ and
any $\epsilon>0$,
there is a sequence $x_1, \ldots, x_n\in A$ satisfying
\[\|\sum_{i=1}^n x_i a x_i^\ast\|\leq 1 \qquad {\rm~and~} \qquad \| \sum_{i=1}^n x_i a x_i^\ast b-b\|<\epsilon \quad {\rm ~for~} b\in F.\]
\end{Lem}
\begin{proof}
Observe that for any sequence $x_1, \ldots, x_n\in A$ and any $b\in F$, the \Cs-norm condition implies
\[\| \sum_{i=1}^n x_i a x_i^\ast b-b\| \leq \| \sum_{i=1}^n x_i a x_i^\ast c-c\|,\]
where $c:= \left(\sum_{d\in F}d d^\ast \right)^{1/2}$.
Therefore we only need to show the statement when $F$ is a singleton in $A_+$.
By the fullness of $a$, we may further assume that the element $b$ in
$F$ is of the form $\sum_{i=1}^n y_i a z_i$; $y_1, \ldots, y_n, z_1, \ldots, z_n \in A$.
In this case, we have
\[b^2 = b^\ast b =\sum_{i, j=1}^n z_i^\ast a y_i^\ast y_j a z_j
\leq C \sum_{i=1}^n z_i^\ast a z_i,\]
where $C:= \| a\| \|(y_i^\ast y_j)_{1\leq i, j \leq n}\|_{\mathbb{M}_n(A)}$.
Put $w:= \left(\sum_{i=1}^n z_i^\ast a z_i \right)^{1/2}$.
Choose a sequence $(f_k)_{k=1}^\infty$ in $C_0(]0, \infty[)_+$
satisfying $t f_k(t)\leq 1$ for all $k\in \mathbb{N}$ and all $t\in [0, \infty[$, and
$\lim_{k \rightarrow \infty} tf_k(t)=1$ uniformly on compact subsets of $]0, \infty[$.
Then, for each $k\in \mathbb{N}$, we have
\[ \sum_{i=1}^n f_k(w)z_i^\ast a z_i f_k(w)
= f_k(w)^2 w^2\leq 1,\]
\[\| \sum_{i=1}^n f_k(w) z_i^\ast a z_i f_k(w) b-b\| \leq \sqrt{C} \| ( w^2 f_k(w)^2-1)w\|.\]
The last term tends to zero as $k\rightarrow \infty$.
Therefore, for a sufficiently large $N\in \mathbb{N}$,
the sequence $(f_N(w)z_i^\ast)_{i=1}^n$ satisfies the required conditions.
\end{proof}
For simple \Cs-algebras, one can strengthen Lemma \ref{Lem:full} as follows.
\begin{Lem}\label{Lem:simple}
Let $A$ be a simple \Cs-algebra. Let $a\in A_+ \setminus \{0\}$.
Then for any $b\in A_+$ and
any $\epsilon>0$,
there is a sequence $x_1, \ldots, x_n\in A$ satisfying
\[\|\sum_{i=1}^n x_i x_i^\ast\| \leq \|a\|^{-1}\|b\|,\qquad \sum_{i=1}^n x_i a x_i^\ast \approx_\epsilon b.\]
\end{Lem}
\begin{proof}
We may assume $\|a\|=\|b\|=1$.
Take $f \in C([0, 1])_+$
satisfying $f(1)\neq 0$ and $\supp(f)\subset [1- \epsilon/2, 1]$.
Then
\[f(a)\neq 0,\qquad \left(1- \frac{\epsilon}{2}\right) f(a)^2 \leq f(a) a f(a) \leq f(a)^2.\]
Applying Lemma \ref{Lem:full} to $f(a)^2$ and $F=\{b^{1/2}\}$, we obtain a sequence
$y_1, \ldots, y_n \in A$ satisfying
\[\|\sum_{i=1}^n y_i f(a)^2 y_i^\ast \|\leq 1,\qquad \|\sum_{i=1}^n y_i f(a)^2 y_i^\ast b^{\frac{1}{2}}-b^{\frac{1}{2}}\|< \frac{\epsilon}{2}.\]
Set $x_i:=b^{1/2}y_i f(a) $ for $i=1, \ldots, n$.
Then
\[\sum_{i=1}^n x_ix_i^\ast  = \sum_{i=1}^n b^{\frac{1}{2}} y_i f(a)^2 y_i^\ast b^{\frac{1}{2}} \leq 1.\]
Straightforward estimations show that
\[\sum_{i=1}^n x_i a x_i^\ast \approx_{ \epsilon/2} \sum_{i=1}^n b^{\frac{1}{2}} y_i f(a)^2 y_i^\ast b^{\frac{1}{2}} \approx_{ \epsilon/2} b.\]
Therefore $x_1, \ldots, x_n$ form the desired sequence.
\end{proof}

As an application of Lemma \ref{Lem:simple}, one can remove the
unital condition from the tensor splitting theorem \cite{Zac}, \cite{Zsi} (cf.~\cite{GK}).
For a \Cs-subalgebra $C$ of $A\otimes B$,
we define the subset $\mathcal{S}_A(C)$ of $B$ to be
\[\mathcal{S}_A(C):=\left\{(\varphi \otimes \id_B)(c): \varphi \in A^\ast, c\in C\right\}.\]
\begin{Thm}\label{Thm:TS}
Let $A$ be a simple \Cs-algebra and $B$ be a \Cs-algebra.
Let $C$ be a \Cs-subalgebra of $A\otimes B$ closed under multiplications by $A$.
Then $\mathcal{S}_A(C)$ forms a \Cs-subalgebra of $C$ and
satisfies
$A \otimes \mathcal{S}_A(C) \subset C$.
Thus, when $A$ satisfies the strong operator approximation property \cite{HK}
or when the inclusion $\mathcal{S}_A(C) \subset B$ admits a completely bounded projection,
we have
$C= A\otimes \mathcal{S}_A(C)$.
\end{Thm}
\begin{proof}
To show the first statement,
it suffices to show the following claim.
For any pure state $\varphi$ on $A$ and
any $a\in A_+$, $c\in C$, with $b := (\varphi \otimes \id_B)(c)$,
we have $a\otimes b \in C$.
Indeed the claim implies that, since the set of pure states on $A$ spans
a weak-$\ast$ dense subspace of $A^\ast$ and $A_+$ spans $A$, for any $a\in A$
and any $\psi \in A^\ast$ with $\psi(a)\neq 0$, the subspace $X:=(\psi \otimes \id_B)(C)$ of $B$
satisfies $a \otimes X =(a\otimes B) \cap C$.
This implies $X=\mathcal{S}_A(C)$, and proves the first statement.
To show the claim, for any $\epsilon>0$, 
by the Akemann--Anderson--Pedersen excision theorem \cite{AAP} (Theorem 1.4.10 in \cite{BO}),
one can take $e\in A_+$ with $\|e\|=1$,
$e c e \approx_{\epsilon} e^2\otimes b$.
By Lemma \ref{Lem:simple}, there is a sequence $x_1, \ldots, x_n \in A$
satisfying
$\sum_{i=1}^n x_ie c ex_i^\ast \approx_{\epsilon} a\otimes b$.
The left term is contained in $C$ by assumption.
Thus $a\otimes b \in C$.

For the last statement, when $A$ satisfies the strong operator approximation property,
the claim follows from Theorem 12.4.4 in \cite{BO}.
When we have a completely bounded projection
$P\colon B \rightarrow \mathcal{S}_A(C)$,
it is not hard to see that for any $\varphi \in A^\ast$,
$(\varphi\otimes \id_B)((\id_A \otimes P)(c)-c)=0$ for all $c\in C$.
This proves $(\id_A \otimes P)|_C=\id_C$ and thus $C\subset A \otimes \mathcal{S}_A(C)$.
\end{proof}
\subsection*{Acknowledgements}
Parts of the present work are greatly improved during the author's visiting
in Research Center for Operator Algebras (Shanghai) for the conference ``Special Week on Operator Algebras 2019''.
He is grateful to the organizers of the conference for kind invitation.
This work was supported by JSPS KAKENHI Early-Career Scientists
(No.~19K14550) and tenure track funds of Nagoya University.
Finally, he would like to thank the second reviewer for helpful comments
which improve some explanations of the article.

\end{document}